\documentclass[12pt]{article}

\usepackage{bm, amsmath, amsthm, amssymb}
\usepackage[margin=1in]{geometry}

\theoremstyle{plain}
\newtheorem{thm}{Theorem}[section]
\newtheorem{lem}[thm]{Lemma}
\newtheorem{prop}[thm]{Proposition}
\newtheorem{cor}[thm]{Corollary}

\theoremstyle{definition}

\theoremstyle{remark}

\newcommand{\qbinom}[2]{\genfrac{[}{]}{0pt}{}{#1}{#2}_q}
\newcommand{\qbinomf}[2]{\genfrac{<}{>}{0pt}{}{#1}{#2}_q}
\newcommand{\F}{\mathbb{F}}
\newcommand{\Z}{\mathbb{Z}}
\newcommand{\Q}{\mathbb{Q}}
\newcommand{\N}{\mathbb{N}}
\newcommand{\C}{\mathbb{C}}

\title{Counting Quiver Representations over Finite Fields Via Graph Enumeration}

\author{Geir T. Helleloid \\ {\normalsize Department of Mathematics}
\\ {\normalsize The University of Texas at Austin} \\ 1 {\normalsize
University Station C1200} \\ {\normalsize Austin, TX 78712-0257} \\
{\normalsize \texttt{geir@math.utexas.edu}} \and Fernando
Rodriguez-Villegas \\ {\normalsize Department of Mathematics} \\
{\normalsize The University of Texas at Austin} \\ 1 {\normalsize
University Station C1200} \\ {\normalsize Austin, TX 78712-0257} \\
{\normalsize \texttt{villegas@math.utexas.edu}}} 

\begin{document}

\maketitle

\begin{abstract}
  Let $\Gamma$ be a quiver on $n$ vertices $v_1, v_2, \dots, v_n$ with
  $g_{ij}$ edges between $v_i$ and $v_j$, and let $\bm{\alpha} \in
  \N^n$.  Hua gave a formula for $A_{\Gamma}(\bm{\alpha}, q)$, the
  number of isomorphism classes of absolutely indecomposable
  representations of $\Gamma$ over the finite field $\F_q$ with
  dimension vector $\bm{\alpha}$.  Kac showed that
  $A_{\Gamma}(\bm{\alpha}, q)$ is a polynomial in $q$ with integer
  coefficients.  Using Hua's formula, we show that for each integer $s
  \ge 0$, the $s$-th derivative of $A_{\Gamma}(\bm{\alpha},q)$ with
  respect to $q$, when evaluated at $q = 1$, is a polynomial in the
  variables $g_{ij}$, and we compute the highest degree terms in this
  polynomial.  Our formulas for these coefficients depend on the
  enumeration of certain families of connected graphs.
\end{abstract}

\section{Introduction}
\label{introsec}

Let $\Gamma$ be a quiver on $n$ vertices $v_1, v_2, \dots, v_n$ with
$g_{ij}$ edges between vertices $v_i$ and $v_j$ for $1 \le i \le j \le
n$.  All of the following results are independent of the orientation
of these edges.  Let $\bm{\alpha} = (\alpha_1, \alpha_2, \dots,
\alpha_n) \in \N^n$ (throughout the paper, vectors will be represented
by boldface symbols).  We are interested in $A_{\Gamma}(\bm{\alpha},
q)$, the number of isomorphism classes of absolutely indecomposable
representations of $\Gamma$ over the finite field $\F_q$ with
dimension vector $\bm{\alpha}$.  Kac~\cite{kac} proved that
$A_{\Gamma}(\bm{\alpha}, q)$ is a polynomial in $q$ with integer
coefficients and that it is independent of the orientation of
$\Gamma$. He conjectured that the coefficients of
$A_{\Gamma}(\bm{\alpha}, q)$ are non-negative and that if $\Gamma$ has
no loops, then the constant term of $A_{\Gamma}(\bm{\alpha}, q)$ is
equal to the multiplicity of $\bm{\alpha}$ in the Kac-Moody algebra
defined by $\Gamma$.  Both conjectures are true for quivers of finite
and tame type (see Crawley-Boevey and Van den Bergh \cite{CBVB}); a
proof of the multiplicity statement in Kac's conjectures for general
quivers was recently announced by Hausel \cite{Hausel}.

Our goal is to understand $A_{\Gamma}(\bm{\alpha}, 1)$, and more
generally $\left. \left( \frac{d^s}{dq^s} \; A_{\Gamma}(\bm{\alpha},
    q) \right) \right|_{q=1}$, as a function of the variables
$g_{ij}$.  The main impetus for studying $A_{\Gamma}(\bm{\alpha}, 1)$
comes from the work of Hausel and Rodriguez-Villegas~\cite{HRV}. They
show that when $\Gamma$ is the quiver $S_g$ consisting of one vertex $v$ with $g$
self-loops, $A_{S_g}(\alpha, 1)$ (where $\bm{\alpha} = \alpha \in \Z$) is
(conjecturally) the dimension of the middle cohomology group of a
character variety parameterizing certain representations of the
fundamental group of a closed genus-$g$ Riemann surface to
${\rm GL}_n(\C)$.

One can imagine that specializing to $q = 1$ will relate $A_\Gamma(\bm{\alpha}, q)$ to
counting representations of $\Gamma$ in the category of finite sets;
this hope follows a well-known philosophy about the significance of
letting $q \to 1$ in formulas that depend on a finite field $\F_q$,
although it seems hard to make this philosophy precise.  In this paper
we show in Theorems~\ref{zeroderivativethm} and~\ref{leadingcoeffthm}
that $\left. \left( \frac{d^s}{dq^s} \; A_{\Gamma}(\bm{\alpha}, q)
  \right) \right|_{q=1}$ is a polynomial in the variables $g_{ij}$,
and we give a formula for its leading coefficients.  This formula
relies on the number of connected graphs in a family determined by
$\Gamma$ and on Stirling numbers of the second kind, which arise from
derivatives of $q$-binomial coefficients.  The description of the
graphs in question is given prior to Theorem~\ref{expformula} and all
necessary information about Stirling numbers and $q$-binomial
coefficients is given in Appendix~\ref{appendixsec}.  Unfortunately,
our proofs of Theorems~\ref{zeroderivativethm}
and~\ref{leadingcoeffthm} do not give any conceptual indication as to
why our results should involve the enumeration of connected graphs.

To illustrate the type of result found in this paper, consider
$\Gamma=S_g$.  Using a formula of Hua~\cite[Theorem 4.6]{hua} for
$A_{\Gamma}(\bm{\alpha}, q)$, which we will present in
Section~\ref{huasec} and which is our starting point for the results
in this paper, we can compute the polynomial $A_{S_g}(\alpha, q)$ for
small $\alpha$ and $g$.  These computations are displayed in the
following table:
\[
\begin{array}{|c|c|c|c|c|}
\hline
A_{S_g}(\alpha, q) & g=1 & g=2 & g=3 & g=4 \\
\hline
\rule[0mm]{0mm}{5mm} \alpha=1 & q & q^2 & q^3 & q^4 \\
\alpha=2 & q & q^5+q^3 & q^9+q^7+q^5 & q^{13}+q^{11}+q^9+q^7 \\
\alpha=3 & q & q^{10}+q^8+q^7+\cdots & q^{19}+q^{17}+q^{16}+\cdots &
q^{28}+q^{26}+q^{25}+\cdots \\
\alpha=4 & q & q^{17}+q^{15}+q^{14}+\cdots & q^{33}+q^{31}+q^{30}+\cdots &
q^{49}+q^{47}+q^{46}+\cdots \\  
\alpha=5 & q & q^{26}+q^{24}+q^{23}+\cdots & q^{51}+q^{49}+q^{48}+\cdots &
q^{76}+q^{74}+q^{73}+\cdots \\  
\alpha=6 & q & q^{37}+q^{35}+q^{34}+\cdots & q^{73}+q^{71}+q^{70}+\cdots &
q^{109}+q^{107}+q^{106}+\cdots \\  
\hline
\end{array}
\]
Evaluating each polynomial at $q = 1$ gives the following values for $A_{S_g}(\alpha, 1)$: 
\[
\begin{array}{|c|c|c|c|c|c|c|}
\hline
A_{S_g}(\alpha, 1) & g=1 & g=2 & g=3 & g=4 & g=5 & g=6 \\
\hline
\rule[0mm]{0mm}{5mm} \alpha=1 & 1 & 1 & 1 & 1 & 1 & 1 \\
\alpha=2 & 1 & 2 & 3 & 4 & 5 & 6 \\
\alpha=3 & 1 & 6 & 15 & 28 & 45 & 66 \\
\alpha=4 & 1 & 22 & 95 & 252 & 525 & 946 \\
\alpha=5 & 1 & 95 & 710 & 2674 & 7215 & 15961 \\
\alpha=6 & 1 & 449 & 5856 & 31374 & 109707 & 298023 \\
\hline
\end{array}
\]
Fitting each row of the above table to a polynomial gives empirical
evidence that the next table is correct: 
\[
\begin{array}{|c|l|}
\hline
& A_{S_g}(\alpha, 1) \\
\hline
\rule[0mm]{0mm}{5mm} \alpha=1 & 1 \\
\rule[0mm]{0mm}{5mm} \alpha=2 & \binom{g}{1} \\
\rule[0mm]{0mm}{5mm} \alpha=3 & 4 \binom{g}{2} + \binom{g}{1} \\
\rule[0mm]{0mm}{5mm} \alpha=4 & 32 \binom{g}{3} + 20 \binom{g}{2} + \binom{g}{1}\\
\rule[0mm]{0mm}{5mm} \alpha=5 & 400 \binom{g}{4} + 428 \binom{g}{3} + 93 \binom{g}{2} + \binom{g}{1} \\
\rule[-2mm]{0mm}{7mm} \alpha=6 & 6912\binom{g}{5} + 10640 \binom{g}{4} +4512 \binom{g}{3} + 447 \binom{g}{2} + \binom{g}{1} \\
\hline
\end{array}
\]
This suggests that $A_{S_g}(\alpha, 1)$ is a polynomial in $g$ of
degree $\alpha-1$ with leading coefficient $2^{\alpha-1}
\alpha^{\alpha-2} / \alpha!$.  We prove this and a generalization to all quivers in
Theorem~\ref{zeroderivativethm} below.  Theorem~\ref{leadingcoeffthm}
offers a similar result for any derivative (with respect to $q$) of
$A_{\Gamma}(\bm{\alpha}, q)$ evaluated at $q = 1$.  

The fact that the leading coefficient of $A_{S_g}(\alpha, 1)$ equals
$2^{\alpha-1} \alpha^{\alpha-2} / \alpha!$ was mentioned (without
proof) in \cite[Remark 4.4.6]{HRV}.  As mentioned above, in the
context of that paper, $S_g$ corresponds to a closed Riemann
surface of genus $g$ and it seems more appropriate to use its Euler
characteristic $2g-2$ instead of $g$ as a variable.  One possibly
telling feature of this choice is that the factor $2^{\alpha-1}$ in
the leading coefficient disappears, though we do not know of a similar
approach for the general case.  Finally, we note that
$\alpha^{\alpha-2}$ appears in the formula for the leading coefficient
of $A_{S_g}(\alpha, 1)$ because $\alpha^{\alpha-2}$ is the number of
trees on $\alpha$ labeled vertices by Cayley's Theorem.  As indicated
above, for other quivers, the leading coefficient formula involves the
enumeration of other families of graphs.

\paragraph{Acknowledgements.} FRV was supported by NSF grant
DMS-0200605. We would like to thank Keith Conrad for his proof of
Theorem~\ref{expthm}.

\section{Hua's Formula}
\label{huasec}

We begin with a presentation of Hua's formula
for $A_{\Gamma}(\bm{\alpha}, q)$.  Let $\bm{T} = (T_1, T_2, \dots,
T_n)$ be a vector of indeterminates.  Let $\mathcal{P}$ denote the set
of all integer partitions, including the unique partition of 0.  If
$\lambda$ and $\mu$ are partitions with transposes $\lambda'$ and
$\mu'$ respectively, let 
\[
\left< \lambda, \mu \right> := \sum_{1 \le i}{\lambda_i' \mu_i'}.
\]
Also, let
\[
b_{\lambda}(q) := \prod_{1 \le i}{\prod_{1 \le j \le n_i}{(1-q^j)}},
\]
where $\lambda$ has $n_i$ parts of size $i$ for each $i$.  As a
notational convenience, we will write monomials as a vector with a
vector exponent, as in $\bm{T}^{\bm{\alpha}} = T_1^{\alpha_1} \cdots
T_n^{\alpha_n}$.  If $\lambda$ is a partition or a
composition, let $|\lambda|$ denote the sum of the parts of $\lambda$.

Finally, define the function $P_{\Gamma}(\bm{T}, q)$ by
\begin{eqnarray}
\label{pdef}
P_{\Gamma}(\bm{T}, q) := \sum_{\lambda^1, \dots, \lambda^n \in
\mathcal{P}}{ \frac{\prod_{1 \le i 
\le j \le n}{q^{g_{ij} \left< \lambda^i, \lambda^j \right>}}}{\prod_{1
\le i \le 
n}{q^{\left< \lambda^i, \lambda^i \right>} b_{\lambda^i}(q^{-1})}} \;
T_1^{\left|\lambda^1\right|} \cdots 
T_n^{\left|\lambda^n\right|}}
\end{eqnarray}
and the function $H_{\Gamma}(\bm{\alpha}, q)$ implicitly by
\begin{eqnarray}
\label{hdef}
\log{P_{\Gamma}(\bm{T}, q)} = \sum_{\bm{0} \neq \bm{\alpha} \in
\N^n}{\frac{H_{\Gamma}(\bm{\alpha}, q)}{\overline{\bm{\alpha}}} \; 
\bm{T}^{\bm{\alpha}}},
\end{eqnarray}
where $\overline{\bm{\alpha}} = \gcd(\alpha_1, \dots, \alpha_n)$.
Hua expresses $A_{\Gamma}(\bm{\alpha}, q)$ in terms of $H_{\Gamma}(\bm{\alpha},q)$.  
\begin{thm}[{Hua~\cite[Theorem 4.6]{hua}}]
\label{huathm}
 \begin{eqnarray}
  \label{aformula}
  A_{\Gamma}(\bm{\alpha}, q) = \frac{q-1}{\overline{\bm{\alpha}}} \sum_{d |
\overline{\bm{\alpha}}}{\mu(d) H_{\Gamma}(\bm{\alpha}/d, q^d)}.
 \end{eqnarray}
\end{thm}

\section{A Deformation of $A_{\Gamma}(\bm{\alpha}, q)$}
\label{graphsec}

Although we want to understand $A_{\Gamma}(\bm{\alpha}, 1)$, we cannot
use Equations~\eqref{pdef},~\eqref{hdef}, and~\eqref{aformula} directly,
since the summands in $P_{\Gamma}(\bm{T}, q)$ have poles at $q = 1$.
We will proceed instead by introducing extra variables, computing
certain limits as $q$ approaches 1, and then specializing the results.
This section analyzes $A_{\Gamma}(\bm{\alpha}, \bm{u}, q)$, a
generalization of $A_{\Gamma}(\bm{\alpha}, q)$, while
Sections~\ref{mahlersec} and~\ref{derivativesec} apply the results to
$A_{\Gamma}(\bm{\alpha}, q)$. 

In what follows, vectors $\bm{u} \in \N^{n(n+1)/2}$ will have
components $u_{ij}$ for $1 \le i \le j \le n$, and for  $\bm{\ell} \in
\N^n$ we let  $\bm{u}^{\bm{\ell}}:=\prod_{1 \le i \le j \le
  n}u_{ij}^{\ell_i \ell_j}$.  
 Let $\bm{u} \in \N^{n(n+1)/2}$. Define functions $P_{\Gamma}(\bm{T},
 \bm{u}, q)$, 
$H_{\Gamma}(\bm{\alpha}, \bm{u}, q)$, and $A_{\Gamma}(\bm{\alpha},
\bm{u}, q)$ by the formulas 
\begin{eqnarray}
 P_{\Gamma}(\bm{T}, \bm{u}, q) &:=& \sum_{\lambda^1, \dots, \lambda^n
 \in \mathcal{P}}{ \frac{\prod_{1 
\le i \le j \le n}{u_{ij}^{\left< \lambda^i, \lambda^j
\right>}}}{\prod_{1 \le i \le 
n}{q^{\left< \lambda^i, \lambda^i \right>} b_{\lambda^i}(q^{-1})}} \;
T_1^{|\lambda^1|} \cdots 
T_n^{|\lambda^n|}}, \label{pgendef} \\ 
\log{P_{\Gamma}(\bm{T}, \bm{u}, q)} &:=& \sum_{\bm{0} \neq \bm{\alpha}
\in \N^n}{\frac{H_{\Gamma}(\bm{\alpha}, \bm{u},
q)}{\overline{\bm{\alpha}}} \; \bm{T}^{\bm{\alpha}}}, \; \textrm{and}
\label{hgendef} 
\\
A_{\Gamma}(\bm{\alpha}, \bm{u}, q) &:=&
\frac{q-1}{\overline{\bm{\alpha}}} \sum_{d |
\overline{\bm{\alpha}}}{\mu(d) H_{\Gamma}(\bm{\alpha}/d, \bm{u}^d,
q^d)}. \label{agendef} 
\end{eqnarray} 
Observe that $P_{\Gamma}(\bm{T},\bm{u}, q)$, $H_{\Gamma}(\bm{\alpha},
\bm{u}, q)$, and $A_{\Gamma}(\bm{\alpha}, \bm{u}, q)$ specialize to
$P_{\Gamma}(\bm{T}, q)$, $H_{\Gamma}(\bm{\alpha}, q)$, and
$A_{\Gamma}(\bm{\alpha}, q)$ respectively  when $u_{ij} = q^{g_{ij}}$
for $1 \le i \le j \le n$.  However, $A_{\Gamma}(\bm{\alpha}, \bm{u},
q)$ typically is not a polynomial in $q$ even though
$A_{\Gamma}(\bm{\alpha}, q)$ is.
For $\bm{\ell} \in \N^n$ let $\bm{\ell}! := \ell_1! \cdots \ell_n!$
and for $\bm{u} \in \N^{n(n+1)/2}$ let $\bm{u}! := u_{11}!
\cdots u_{ij}! \cdots u_{nn}!$.  Our first result computes a limit
involving $A_{\Gamma}(\bm{\alpha}, \bm{u}, q)$. 

\begin{prop}
\label{pregraphthm}
\begin{eqnarray}
\lim_{q \to 1}{(q-1)^{|\bm{\alpha}|-1} A_{\Gamma}(\bm{\alpha}, \bm{u}, q)}
&=& \left[\textrm{the coefficient of } \bm{T}^{\bm{\alpha}} \textrm{
in } \right] \log{\sum_{\bm{\ell} \in \N^n}{ \bm{u}^{\bm{\ell}} \;
\frac{\bm{T}^{\ell}}{\bm{\ell}!}}}. \label{alim} 
\end{eqnarray}
\end{prop}

\begin{proof}
We begin with a limit involving $P_{\Gamma}(\bm{T}, \bm{u}, q)$.  By
the definition of $P_{\Gamma}(\bm{T}, \bm{u}, q)$,  
\begin{eqnarray}
&& \lim_{q \to 1}{P_{\Gamma}((q-1)\bm{T}, \bm{u}, q)} \label{plim} \\ 
&=& \lim_{q \to 1}{\sum_{\lambda^1, \dots, \lambda^n \in
\mathcal{P}}{\frac{\prod_{1 \le 
i \le j \le n}{u_{ij}^{\left< \lambda^i, \lambda^j \right>}}}{\prod_{1
\le i \le 
n}{q^{\left< \lambda^i, \lambda^j \right>} b_{\lambda^i}(q^{-1})}}} \;
(q-1)^{|\lambda^1| + 
\cdots + |\lambda^n|} \; T_1^{|\lambda^1|} \cdots
T_n^{|\lambda^n|}}. \nonumber 
\end{eqnarray}
The quantity $b_{\lambda}(q^{-1})$ has a zero at $q = 1$ of
multiplicity $\ell(\lambda)$ (the number of parts of $\lambda$).  Thus
each summand on the right-hand side of Equation~\eqref{plim} has limit 0
as $q \to 1$ unless $\lambda^i = \left(1^{\ell_i}\right)$ for $1 \le i
\le n$ for some $\bm{\ell}=(\ell_1,\ldots,\ell_n) \in \N^n$.  In this case,  
\[
\lim_{q \to
1}{\frac{(q-1)^{\ell_i}}{b_{\left(1^{\ell_i}\right)}(q^{-1})}} =
\frac{1}{\ell_i!}. 
\]
Therefore
\begin{eqnarray}
 \label{plim2} 
\lim_{q \to 1}{P_{\Gamma}((q-1)\bm{T}, \bm{u}, q)}
&=& \sum_{\bm{\ell} \in \N^n}{\bm{u}^{\bm{\ell}}\;
\frac{\bm{T}^{\bm{\ell}}}{\bm{\ell}!}},
\end{eqnarray}
since
$$
\langle (1^{\ell_i}),(1^{\ell_j})\rangle = \ell_i\ell_j.
$$
Combining this limit with Equation~\eqref{hgendef}, the defining
equation for $H_{\Gamma}(\bm{\alpha}, \bm{u}, q)$, shows that 
\begin{eqnarray}
\label{hlim}
\lim_{q \to 1}{\sum_{\bm{0} \neq \bm{\alpha} \in
\N^n}{\frac{H_{\Gamma}(\bm{\alpha}, \bm{u},
q)}{\overline{\bm{\alpha}}} (q-1)^{|\bm{\alpha}|} \;
\bm{T}^{\bm{\alpha}}}} &=& 
\log{\sum_{\bm{\ell} \in \N^n}{ \bm{u}^{\bm{\ell}}} \;
\frac{\bm{T}^{\bm{\ell}}}{\bm{\ell}!}}. 
\end{eqnarray}
In particular, this proves that
\begin{eqnarray}
\label{hlimexists}
\lim_{q \to 1}{(q-1)^{|\bm{\alpha}|} H_{\Gamma}(\bm{\alpha}, \bm{u}, q)}
\end{eqnarray}
exists and therefore
\begin{eqnarray}
\label{hlimis0}
\lim_{q \to 1}{(q-1)^{|\bm{\alpha}|} H_{\Gamma}(\bm{\alpha}/d, \bm{u}, q)} = 0
\end{eqnarray}
if $d$ is greater than 1 and divides $\overline{\bm{\alpha}}$.
Equations~\eqref{agendef},~\eqref{hlim},~\eqref{hlimexists},
and~\eqref{hlimis0} show that 
\begin{eqnarray}
\lim_{q \to 1}{(q-1)^{|\bm{\alpha}|-1} A_{\Gamma}(\bm{\alpha}, \bm{u}, q)}
&=& \lim_{q \to
1}{\frac{(q-1)^{|\bm{\alpha}|}}{\overline{\bm{\alpha}}} \sum_{d |
\overline{\bm{\alpha}}}{\mu(d) H_{\Gamma}(\bm{\alpha}/d, \bm{u}^d,
q^d)}} \nonumber \\ 
&=& \lim_{q \to
1}{\frac{(q-1)^{|\bm{\alpha}|}}{\overline{\bm{\alpha}}}
H_{\Gamma}(\bm{\alpha}, \bm{u}, q)} \nonumber \\ 
&=& \left[\textrm{the coefficient of } \bm{T}^{\bm{\alpha}} \textrm{ 
in } \right] \log{\sum_{\bm{\ell} \in \N^n}{\bm{u}^{\bm{\ell}} \; 
\frac{\bm{T}^{\ell}}{\bm{\ell}!}}}. \nonumber  
\end{eqnarray}
\end{proof}

\section{Multivariate Exponential Formula}
The limit in Proposition~\ref{pregraphthm}, namely
Equation~\eqref{alim}, can be rewritten using a multivariate version
of the Exponential Formula applied to the enumeration of graphs.  Our
presentation of the Multivariate Exponential Formula
(Theorem~\ref{mvexp}) follows that of the usual Exponential Formula
found in Stanley~\cite[Section 5.1]{sta}.

Let $K$ be any field and let $\bm{X} = (X_1, X_2, \dots, X_n)$ be a
vector of indeterminates.  Given a function $f : \N^n \to K$, define  
\begin{eqnarray}
E_f(\bm{X}) &:=& \sum_{\ell \in \N^n}{f(\ell_1, \dots, \ell_n) \;
  \frac{\bm{X}^{\ell}}{\bm{\ell}!}}. 
\end{eqnarray}
Similarly, given a function $f : \N^n \setminus \{\bm{0}\} \to K$,
define
\begin{eqnarray}
  E_f(\bm{X}) &:=& \sum_{\ell \in \N^n \setminus \{\bm{0}\}}{f(\ell_1,
    \dots, \ell_n) \; \frac{\bm{X}^{\ell}}{\bm{\ell}!}}. 
\end{eqnarray}
Also, if $V_1, \dots, V_n$ are finite (disjoint) sets, let $\Pi(V_1,
\dots, V_n)$ denote the set of set partitions of $V_1 \cup \cdots \cup
V_n$.  If $\pi \in \Pi(V_1, \dots, V_n)$, let $\ell(\pi)$ equal the
number of blocks in $\pi$ and let $\pi_1, \dots, \pi_{\ell(\pi)}$
denote the blocks of $\pi$.
\begin{thm}[Multivariate Exponential Formula] \label{mvexp} 
Let $f : \N^n \setminus \{\bm{0}\} \to K$.  Define a function $g :
\N^n \to K$ by $g(0, \dots, 0) := 1$ and by 
\begin{eqnarray}
g(|V_1|, \dots, |V_n|) &:=& \sum_{\pi \in \Pi(V_1, \dots,
  V_n)}{\prod_{i=1}^{\ell(\pi)}{f(|\pi_i \cap V_1|, \dots, |\pi_i \cap
    V_n|)}} \label{mvformula} 
\end{eqnarray}
if at least one of the sets $V_i$ is nonempty.  Then
$\log{E_g(\bm{X})} = E_f(\bm{X})$. 
\end{thm}

\begin{proof}
Fix a positive integer $k$ and define the function $g_k \; : \; \N^n
\setminus \{\bm{0}\} \to K$ by 
\begin{eqnarray}
  g_k(|V_1|, \dots, |V_n|) &:=& \sum_{\genfrac{}{}{0pt}{1}{\pi \in
      \Pi(V_1, \dots, V_n)}{\ell(\pi) = k}}{\prod_{i=1}^{k}{f(|\pi_i
      \cap V_1|, \dots, |\pi_i \cap V_n|)}}. \nonumber 
\end{eqnarray}
Since $\pi_1, \dots, \pi_k$ are non-empty, they are all distinct and
there are $k!$ ways of ordering them.  Thus 
\[
E_{g_k}(\bm{X}) = \frac{1}{k!} \; E_f(\bm{X})^k.
\]
Then
\begin{eqnarray*}
E_g(\bm{X}) &=& 1+\sum_{k=1}^{\infty}{E_{g_k}(\bm{X})} \\
&=& 1+\sum_{k=1}^{\infty}{\frac{1}{k!} \; E_f(\bm{X})^k} \\
&=& e^{E_f(\bm{X})},
\end{eqnarray*}
proving the theorem.
\end{proof}

To apply Theorem~\ref{mvexp} to the enumeration of graphs, we must
introduce some more notation.  If $\bm{\ell} \in \N^n$, let
$\mathcal{G}^{\bm{\ell}}$ be the set of graphs on the vertices $v_1,
v_2, \dots, v_{|\bm{\ell}|}$.  Let $V_1 := \{v_1, \dots, v_{\ell_1}\},
V_2 := \{v_{\ell_1+1}, \dots, v_{\ell_2}\}, \dots, V_n :=
\{v_{|\bm{\ell}|-\ell_n+1}, \dots, v_{|\bm{\ell}|}\}$.  If $\bm{k} \in
\N^{n(n+1)/2}$, then let $\mathcal{G}^{\bm{\ell}}_{\bm{k}}$ be the set
of graphs in $\mathcal{G}^{\bm{\ell}}$ that have $k_{ij}$ edges
between $V_i$ and $V_j$ for $1 \le i \le j \le n$ and let
$G_{\bm{k}}^{\bm{\ell}}$ be the number of connected graphs in
$\mathcal{G}^{\bm{\ell}}_{\bm{k}}$.  Since a connected graph in $\mathcal{G}^{\bm{\ell}}$ must have at
least $|\bm{\ell}|-1$ edges, the following proposition is clear.

\begin{prop}
\label{conngraphsprop}
If $|\bm{k}| < |\bm{\ell}|-1$, then $G_{\bm{k}}^{\bm{\ell}}=0$.
\end{prop}

Now let $\bm{x} = (x_{11}, \dots, x_{ij}, \dots, x_{nn})$ be a vector of $n(n+1)/2$
indeterminates, where $1 \le i \le j \le n$, and define the weight of
$G \in \mathcal{G}^{\bm{\ell}}_{\bm{k}}$ to be
$\bm{x}^{\bm{k}}:=\prod_{1 \le i \le j \le n}x_{ij}^{k_{ij}}$.

\begin{thm} \label{expformula}
\begin{eqnarray}
&& \log{\left( \sum_{\bm{\ell} \in \N^n}{\left( \prod_{1 \le i < j \le
n}{(1+x_{ij})^{\ell_i \ell_j}}\right)\left( \prod_{1 \le i \le
n}{(1+x_{ii})^{\binom{\ell_i}{2}}}\right) \;
\frac{\bm{X}^{\bm{\ell}}}{\bm{\ell}!}} \right)} \label{graphenum} \\ 
&=& \sum_{0 \neq \bm{\alpha} \in \N^n}{\sum_{\bm{k} \in
\N^{n(n+1)/2}}{G_{\bm{k}}^{\bm{\alpha}} \bm{x}^{\bm{k}} \;
\frac{\bm{X}^{\bm{\alpha}}}{\bm{\alpha}!}}}, \nonumber 
\end{eqnarray}
where $G_{\bm{k}}^{\bm{\alpha}}$ is the number of connected graphs in
$\mathcal{G}^{\bm{\alpha}}_{\bm{k}}$. 
\end{thm}

\begin{proof}
Let $K = \C(x_{11}, \dots, x_{ij}, \dots, x_{nn})$.  Let
\[
\begin{array}{rcccl}
f &:& \N^n \setminus \{\bm{0}\} &\to& K \\
&& (\alpha_1, \dots, \alpha_n) &\mapsto& \displaystyle \sum_{k \in \N^{n(n+1)/2}}{G_{\bm{k}}^{\bm{\alpha}} \bm{x}^{\bm{k}}} \\
&&&& \\
g &:& \N^n &\to& K \\
&& (\ell_1, \dots, \ell_n) &\mapsto& \displaystyle \prod_{1 \le i \le j \le n}{(1+x_{ij})^{\ell_i \ell_j}} \prod_{1 \le i \le n}{(1+x_{ii})^{\binom{\ell_{i}}{2}}}.
\end{array}
\]
The sum of the weights of the connected graphs in
$\mathcal{G}^{\bm{\alpha}}$ equals $f(\alpha_1, \dots, \alpha_n)$.
The sum of the weights of the graphs in $\mathcal{G}^{\bm{\ell}}$
equals $g(\ell_1, \dots, \ell_n)$.  A graph in
$\mathcal{G}^{\bm{\ell}}$ is obtained by choosing a partition of the
vertex set and choosing a connected graph structure for each block of
the partition (and the weight of such a graph is the product of the
weights of its connected components).  It is clear that $f$ and $g$
satisfy Equation~\eqref{mvformula}, so Theorem~\ref{mvexp} proves the
desired result. 
\end{proof}

Perhaps Theorem~\ref{mvexp} and its application in
Theorem~\ref{expformula} to the enumeration of graphs are known, but
we do not know of any reference and hence have included the proofs.
Presumbly there is no explicit formula for $G_{\bm{k}}^{\bm{\alpha}}$
in general.  However, when $|\bm{k}| = |\bm{\alpha}|-1$ (that is, when
the connected graphs in $\mathcal{G}_{\bm{k}}^{\bm{\alpha}}$ are
trees), certain sums of the numbers $G_{\bm{k}}^{\bm{\alpha}}$ can be
computed by the methods in Knuth~\cite{knu}.

\section{$A_\Gamma(\bm{\alpha}, \bm{u}, q)$ and Connected Graphs}

To understand Equation~\eqref{alim} better, we obtain a corollary of
Theorem~\ref{expformula} by rewriting Equation~\eqref{graphenum} with
the substitutions $1 + x_{ij} = u_{ij}$ ($1 \le i < j \le n$), $1 +
x_{ii} = u_{ii}^2$ ($1 \le i \le n$), and $X_i = u_{ii} T_i$ ($1 \le i
\le n$).
\begin{cor}
\label{expcor}
\begin{eqnarray}
&& \log{\sum_{\bm{\ell} \in \N^n}{\bm{u}^{\bm{\ell}} \;
\frac{\bm{T}^{\bm{\ell}}}{\bm{\ell}!}}} \label{expformulamod} \\  
&=& \sum_{\bm{0} \neq \bm{\alpha} \in \N^n}{\sum_{\bm{k} \in
\N^{n(n+1)/2}}{G_{\bm{k}}^{\bm{\alpha}}\left( \prod_{1 \le i < j \le
n}{(u_{ij}-1)^{k_{ij}}} \right) \left( \prod_{1 \le i \le
n}{(u_{ii}^2-1)^{k_{ii}}} \right) \left( \prod_{1 \le i \le
n}{u_{ii}^{\alpha_i}} \right) \;
\frac{\bm{T}^{\bm{\alpha}}}{\bm{\alpha}!}}}. \nonumber 
\end{eqnarray} 
\end{cor}
This allows us to rewrite the result of Proposition~\ref{pregraphthm}.
For each $\bm{k} \in \N^{n(n+1)/2}$, let 
\begin{equation}
\label{sdef}
S_{\bm{k}} = \left\{ \bm{p} \in \N^{n(n+1)/2} \; : \;
  \begin{array}{l}
    k_{ii} \ge p_{ii} \textrm{ for } 1 \le i \le n \\
    k_{ij} = p_{ij} \textrm{ for } 1 \le i < j \le n
  \end{array}
    \right\}.
\end{equation}

\begin{prop}
\label{intermediatethm}
For each $\bm{k} \in \N^{n(n+1)/2}$,
\begin{eqnarray}
&& \left[ \textrm{the coefficient of } (\bm{u}-\bm{1})^{\bm{k}}  \textrm{ in }
\right] \lim_{q \to 
1}{(q-1)^{|\bm{\alpha}|-1} A_{\Gamma}(\bm{\alpha}, \bm{u}, q)}
\label{alimcoef} \\ 
&=& \frac{1}{\bm{\alpha}!} \sum_{\bm{p} \in
S_{\bm{k}}} c^{\bm{\alpha}}_{\bm{k}\bm{p}}\; {G_{\bm{p}}^{\bm{\alpha}}}, 
\nonumber
\end{eqnarray}
where
\[
c^{\bm{\alpha}}_{\bm{k}\bm{p}}:= \prod_{1 \le i \le n}{\left(
\sum_{j=0}^{\infty}{\binom{p_{ii}}{j} \binom{\alpha_i}{k_{ii} - p_{ii}
- j} 2^{p_{ii}-j}} \right)}
\]
and
\[
(\bm{u}-\bm{1})^{\bm{k}}:=\prod_{1\le i \le j \le n}(u_{ij}-1)^{k_{ij}}.
\]
In particular, if $|\bm{k}|=|\bm{\alpha}|-1$, then
\begin{eqnarray}
&& \left[ \textrm{the coefficient of } (\bm{u}-\bm{1})^{\bm{k}}
\textrm{ in } 
\right] \lim_{q \to 
1}{(q-1)^{|\bm{\alpha}|-1} A_{\Gamma}(\bm{\alpha}, \bm{u}, q)}
\label{alimcoef-1} \\ 
&=& \frac{1}{\bm{\alpha}!}2^{t(\bm{k})} \; {G_{\bm{k}}^{\bm{\alpha}}}, 
\nonumber
\end{eqnarray}
where
$$
t(\bm{k}):=\sum_{1 \le i \le n}{k_{ii}}.
$$
\end{prop}

\begin{proof}
Combining Proposition~\ref{pregraphthm} and Corollary~\ref{expcor},
changing variables from $\bm{k}$ to $\bm{p}$ for convenience, and
using the Binomial Theorem shows that 
\begin{eqnarray} 
&& \lim_{q \to 1}{(q-1)^{|\bm{\alpha}|-1} A_{\Gamma}(\bm{\alpha},
\bm{u}, q)} \nonumber \\ 
&=& \frac{1}{\bm{\alpha}!} \sum_{\bm{p} \in
\N^{n(n+1)/2}}{G_{\bm{p}}^{\bm{\alpha}} \left( \prod_{1 \le i < j \le
n}{(u_{ij}-1)^{p_{ij}}} \right) \left( \prod_{1 \le i \le
n}{(u_{ii}^2-1)^{p_{ii}}} \right) \left( \prod_{1 \le i \le
n}{u_{ii}^{\alpha_i}} \right)} \nonumber \\ 
&=& \frac{1}{\bm{\alpha}!} \sum_{\bm{p} \in
\N^{n(n+1)/2}}{G_{\bm{p}}^{\bm{\alpha}} \left( \prod_{1 \le i \le
n}{(u_{ij}+1)^{p_{ii}}} \right) \left( \prod_{1 \le i \le
n}{u_{ii}^{\alpha_{ii}}} \right) \left( \prod_{1 \le i \le j \le
n}{(u_{ij}-1)^{p_{ij}}} \right)} \nonumber \\ 
&=& \frac{1}{\bm{\alpha}!} \sum_{\bm{p} \in
\N^{n(n+1)/2}}{G_{\bm{p}}^{\bm{\alpha}} \left( \prod_{1 \le i \le
n}{\sum_{j=0}^{p_{ii}}{\binom{p_{ii}}{j} (u_{ii}-1)^j 2^{p_{ii}-j}}}
\right) } \nonumber \\ 
&& \hspace{1.5in} \cdot \left( \prod_{1 \le i \le
n}{\sum_{j=0}^{\alpha_i}{\binom{\alpha_i}{j} (u_{ii}-1)^j}} \right)
\left(\prod_{1 \le i \le j \le n}{(u_{ij} -
1)^{p_{ij}}}\right). \nonumber 
\end{eqnarray}
Note that $G_{\bm{p}}^{\bm{\alpha}}$ is nonzero for finitely many
$\bm{p}$, so the above sum over $\bm{p}$ is a finite sum.  The
expression  
\[
G_{\bm{p}}^{\bm{\alpha}} \left( \prod_{1 \le i \le
n}{\sum_{j=0}^{p_{ii}}{\binom{p_{ii}}{j} (u_{ii}-1)^j 2^{p_{ii}-j}}}
\right) \left( \prod_{1 \le i \le
n}{\sum_{j=0}^{\alpha_i}{\binom{\alpha_i}{j} (u_{ii}-1)^j}} \right)
\left( \prod_{1 \le i \le j \le n}{(u_{ij} - 1)^{p_{ij}}} \right), 
\] 
when viewed as a polynomial in the variables $u_{ij}-1$ ($1 \le i \le j \le n$), has
a nonzero 
$(\bm{u}-\bm{1})^{\bm{k}}$
term if
$\bm{p} \in S_{\bm{k}}$, and the coefficient of that term is  
\[
G_{\bm{p}}^{\bm{\alpha}} \prod_{1 \le i \le n}{\left(
\sum_{j=0}^{\infty}{\binom{p_{ii}}{j} \binom{\alpha_i}{k_{ii} - p_{ii}
- j} 2^{p_{ii}-j}} \right)}. 
\]
Summing over all $\bm{p} \in S_{\bm{k}}$ and dividing by
$\bm{\alpha}!$ completes the proof of Equation~\eqref{alimcoef}. 

To prove Equation~\eqref{alimcoef-1} it is enough to note that by
Theorem~\ref{conngraphsprop}, the only non-zero summand on the right-hand
side of Equation~\eqref{alimcoef} occurs when $\bm{p}=\bm{k}$ and 
\[
c^{\bm{\alpha}}_{\bm{k}\bm{k}}=2^{t(\bm{k})}.
\]
\end{proof}

Observe that the left- and right-hand sides of
Equation~\eqref{alimcoef} are nonzero for finitely many $\bm{k}$ and
that the sum over $j$ is actually finite by the definition of binomial
coefficients.  Also, the sum over $j$ can be expressed in terms of a
hypergeometric series as 
\[
\sum_{j=0}^{\infty}{\binom{p_{ii}}{j} \binom{\alpha_i}{k_{ii} - p_{ii}
- j} 2^{p_{ii}-j}} 
= 2^{p_{ii}} \binom{\alpha_i}{k_{ii}-p_{ii}}\,_2F_1(-p_{ii}, 
-k_{ii}+p_{ii}; a_i-k_{ii}+p_{ii}+1; 1/2), 
\]
if desired.

\section{A Mahler-type Expansion for $A_{\Gamma}(\bm{\alpha}, \bm{u}, q)$}
\label{mahlersec}

We can use Proposition~\ref{intermediatethm} to understand
$A_{\Gamma}(\bm{\alpha}, q)$ if we rewrite $A_{\Gamma}(\bm{\alpha},
\bm{u}, q)$ using the Mahler-type expansion given in the following
theorem.

\begin{thm}
\label{expthm}
 If $f \in \Q(q)[x_1, \dots, x_r]$ and $f(q^{b_1}, \dots, q^{b_r}) \in \Z[q]$
for all non-negative integers $b_1, \dots, b_r$, then there are polynomials
$\{c_{\bm{\ell}}(q) \in \Z[q] \; : \; \bm{\ell} \in \N^r\}$ such that
\begin{eqnarray}
\label{mahlerexp}
 f &=& \sum_{\bm{\ell} \in \N^r}{c_{\bm{\ell}}(q) \prod_{1 \le i
\le r}{\qbinomf{x_i}{\ell_i}}}, 
\end{eqnarray}
where
\begin{eqnarray}
\label{angledef}
\qbinomf{x}{\ell} &:=& \prod_{1 \le i' \le
  \ell}{\frac{(x/q^{i'-1}-1)}{(q^{i'}-1)}}
\end{eqnarray}
and $c_\ell(q) = 0$ for all but finitely many $\ell$.

\end{thm}

The proof of this theorem was communicated to us by Keith Conrad.

\begin{proof}[Proof by Keith Conrad]

An expression like Equation~\eqref{mahlerexp} exists for any polynomial $f$
for some unique $c_{\bm{\ell}}(q)\in \Q(q)$ since the products
\[
\prod_{1 \le i
\le r}{\qbinomf{x_i}{\ell_i}}
\]
are an additive basis of $\Q(q)[x_1,\ldots,x_r]$.  We now show that
under the hypothesis on $f$ we actually have $c_{\bm{\ell}}(q)\in \Z[q]$.

Let $E$ be the shift operator acting on functions $h$ of 
$b\in\Z_{\geq0}$ by
\[
(Eh)(b):=h(b+1),
\]
and for each non-negative integer $n$, define
\[
\Delta_n:=(E-I)(E-qI)\cdots(E-q^{n-1}I),
\] 
where $I$ is the identity operator.  Since $f(b_1, \dots,
b_r):=F(q^{b_1}, \dots, q^{b_r})$ for non-negative integers $b_1,
\dots, b_r$,
a standard property of the $q$-difference operators $\Delta_i$ yields
\[
c_{\bm{\ell}}(q)=(\Delta_1^{\ell_1}\cdots\Delta_r^{\ell_r}F)(0,\ldots,0),
\]
where ${\bm{\ell}}=(\ell_1,\ldots,\ell_r)$ and $\Delta_i$ acts on the variable
$b_i$. Explicitly, we obtain the following expression
\[
c_{\bm{\ell}}(q)=\sum_{\bm{j} \in \N^r}\prod_{i=1}^r\binom{\ell_i}{j_i}(-1)^{j_i}q^{j_i(j_i-1)/2}
f(q^{\ell_1-j_1},\ldots, q^{\ell_r-j_r}),
\]
which shows that $c_{\bm{\ell}}(q)\in \Z[q]$ since by hypothesis
$f(q^{\ell_1-j_1},\ldots, q^{\ell_r-j_r})\in \Z[q]$ when $j_i \le \ell_i$ for $1 \le i \le r$.
\end{proof}

Note that $\qbinomf{x}{\ell} = \qbinom{b}{\ell}$ when $x = q^b$ and that
\begin{eqnarray}
\label{qbinlim}
\lim_{q \to 1}{(q-1)^{\ell} \qbinomf{x}{\ell}} &=& \frac{(x-1)^{\ell}}{\ell!}.
\end{eqnarray} 
Here $\qbinom{b}{\ell}$ is a $q$-binomial coefficient (see
Appendix~\ref{appendixsec} for more information about $q$-binomial
coefficients).  By Theorem~\ref{expthm} and the fact that
$A_{\Gamma}(\bm{\alpha}, q) \in \Z[q]$, we can write 
\begin{eqnarray}
 \label{aexp}
A_{\Gamma}(\bm{\alpha}, \bm{u}, q) &=& \sum_{\bm{k} \in
\N^{n(n+1)/2}}{a_{\Gamma}(\bm{\alpha}, \bm{k}, q)}\; \qbinomf{\bm{u}}{\bm{k}},
\end{eqnarray}
for some $a_{\Gamma}(\bm{\alpha}, \bm{k}, q) \in \Z[q]$, where 
$$
\qbinomf{\bm{u}}{\bm{k}}:= \prod_{1 \le i \le j \le
n}{\qbinomf{u_{ij}}{k_{ij}}}.
$$
  Hence
\begin{eqnarray}
\label{aqexp}
A_{\Gamma}(\bm{\alpha}, q) &=& \sum_{\bm{k} \in
\N^{n(n+1)/2}}{a_{\Gamma}(\bm{\alpha}, \bm{k}, q)}\;
\qbinom{\bm{g}}{\bm{k}},
\end{eqnarray}
where
$$
\qbinom{\bm{g}}{\bm{k}}:= \prod_{1 \le i \le j \le
n}{\qbinom{g_{ij}}{k_{ij}}}.
$$
It turns out that Proposition~\ref{intermediatethm} leads to a formula for the
derivatives of $a_{\Gamma}(\bm{\alpha}, \bm{k}, q)$ evaluated at $q =
1$, given in Proposition~\ref{coefderivthm}, which in turn produces a
formula for the derivatives of $A_{\Gamma}(\bm{\alpha}, q)$ evaluated
at $q = 1$ (see Theorems~\ref{zeroderivativethm}
and~\ref{leadingcoeffthm}).  

\begin{prop}
\label{coefderivthm}
For $\bm{k}\in \N^{n(n+1)/2}$ such that $|\bm{k}|> |\bm{\alpha}|$ we
have 
\begin{eqnarray}
\left.
 \frac{a_{\Gamma}(\bm{\alpha}, \bm{k}, q)}{(q-1)^{
     |\bm{k}|-|\bm{\alpha}|+1}} 
\right|_{q=1} 
&=& \frac{\bm{k}!}{\bm{\alpha}!} \sum_{\bm{p} \in  
S_{\bm{k}}} c^{\bm{\alpha}}_{\bm{k}\bm{p}} \; G_{\bm{p}}^{\bm{\alpha}}.
\end{eqnarray}
\end{prop}

\begin{proof}
Using Equations~\eqref{qbinlim} and~\eqref{aexp}, we find
\begin{eqnarray}
&& \lim_{q \to 1}{(q-1)^{|\bm{\alpha}|-1} A_{\Gamma}(\bm{\alpha},
\bm{u}, q)} \label{aexpformula} \\ 
&=& \lim_{q \to 1}{\sum_{\bm{k} \in
\N^{n(n+1)/2}}{a_{\Gamma}(\bm{\alpha}, \bm{k}, q)
(q-1)^{|\bm{\alpha}|-1}}} \;\qbinomf{\bm{u}}{\bm{k}}\\ 
&=& \lim_{q \to 1}{\sum_{\bm{k} \in
\N^{n(n+1)/2}}{\frac{a_{\Gamma}(\bm{\alpha}, \bm{k},
q)}{(q-1)^{|\bm{k}|-|\bm{\alpha}|+1}} \; (q-1)^{\bm{k}} \;\qbinomf{\bm{u}}{\bm{k}}}}\\
&=& \sum_{\bm{k} \in \N^{n(n+1)/2}}{
\left. \frac{a_{\Gamma}(\bm{\alpha}, \bm{k}, q)}{(q-1)^{|\bm{k}|-|\bm{\alpha}|+1}}
\right|_{q=1}  \frac{1}{\bm{k}!}(\bm{u}-\bm{1})^{\bm{k}} }. \nonumber 
\end{eqnarray}
The claim now follows from Proposition~\ref{intermediatethm}.
\end{proof}
Note that if $|\bm{k}| \leq |\bm{\alpha}|$, Proposition~\ref{coefderivthm}
says nothing about $a_{\Gamma}(\bm{\alpha}, \bm{k}, q)$ at $q = 1$.
This is why Theorems~\ref{zeroderivativethm} and~\ref{leadingcoeffthm}
below only give information about leading coefficients.  The first
consequence of Proposition~\ref{coefderivthm} appears when we evaluate
$A_{\Gamma}(\bm{\alpha}, 1)$.  By Equation~\eqref{aqexp},
\begin{eqnarray}
\label{A-1-fmla}
A_{\Gamma}(\bm{\alpha}, 1) &=& \sum_{\bm{k} \in
\N^{n(n+1)/2}}{a_{\Gamma}(\bm{\alpha}, \bm{k}, 1) {
\binom{\bm{g}}{\bm{k}}
}}, 
\end{eqnarray}
where
$$
\binom{\bm{g}}{\bm{k}}:=\prod_{1 \le i \le j \le
n}{\binom{g_{ij}}{k_{ij}}}.
$$

\begin{thm}
 \label{zeroderivativethm} 
The quantity $A_{\Gamma}(\bm{\alpha}, 1)$ is a polynomial in the
 variables $g_{ij}$ whose  homogeneous component of highest degree
 $A^*_{\Gamma}(\bm{\alpha}, 1)$ has total degree
 $|\bm{\alpha}|-1$ and has the form
\begin{equation}
\label{zeroderivativecoef} 
A^*_{\Gamma}(\bm{\alpha}, 1)
= \frac 1 {\bm{\alpha}!}\sum_{|\bm{k}| =
 |\bm{\alpha}|-1}  C_{\Gamma,\bm{k}}^{\bm{\alpha}}\, \bm{g}^{\bm{k}},
\end{equation}
where
$$
C_{\Gamma,\bm{k}}^{\bm{\alpha}}:=
2^{t(\bm{k})}\,G_{\bm{k}}^{\bm{\alpha}} \quad \textrm{and} \quad t(\bm{k}):=\sum_{1 \le
  i \le n}{k_{ii}}.
$$
\end{thm}
\begin{proof}
  That $A_{\Gamma}(\bm{\alpha}, 1)$ is a polynomial in the variables
  $g_{ij}$ is clear from Equation~\eqref{A-1-fmla} since that is true
  for the binomials $\binom{\bm{g}}{\bm{k}}$. The binomial
  $\binom{\bm{g}}{\bm{k}}$ has total degree $|\bm{k}|$, and by
  Proposition~\ref{coefderivthm}, $a_{\Gamma}(\bm{\alpha}, \bm{k},
  1)=0$ if $|\bm{k}|>|\bm{\alpha}|-1$.  Hence the total degree of
  $A_{\Gamma}(\bm{\alpha}, 1)$ is $|\bm{\alpha}|-1$.  To finish the
  proof it is enough to combine Equation~$\eqref{alimcoef-1}$ of
  Proposition~\ref{intermediatethm} with Proposition~\ref{coefderivthm}.
\end{proof}

As a special case of this theorem, we can consider the quiver $S_g$
from Section~\ref{introsec}, which has a single vertex (so $n=1$) and
$g$ loops, and $\bm{\alpha} = \alpha$.  In this case,
$A_{\Gamma}(\alpha, 1)$ is a polynomial in $g$ of degree $\alpha - 1$
and leading coefficient $2^{\alpha-1} G^{\alpha}_{\alpha-1}/\alpha!$
But $G^{\alpha}_{\alpha-1}$ is just the number of (spanning) trees on
$\alpha$ labeled vertices, which is $\alpha^{\alpha-2}$ by Cayley's
Theorem.  So the leading coefficient is $2^{\alpha-1}
\alpha^{\alpha-2} / \alpha!$ as claimed in the introduction.

\section{The derivatives $\frac{d^s}{dq^s} \;
A_{\Gamma}(\bm{\alpha}, q) $ at $q = 1$} 
\label{derivativesec} 

We can proceed further by differentiating Equation~\eqref{aqexp} to
obtain information about the highest order terms of the $s$-th
derivative of $A_{\Gamma}(\bm{\alpha}, q)$ evaluated at $q = 1$.  This
is the subject of the next theorem.  Note that
$G_{\bm{p}}^{\bm{\alpha}}$ was defined prior to
Theorem~\ref{expformula}, $S_{\bm{k}}$ was defined in
Equation~\eqref{sdef}, and $S(\ell, k)$ is the Stirling number of the
second kind discussed in Appendix~\ref{appendixsec}.  Also, if
$\bm{k}, \bm{\ell} \in \N^{n(n+1)/2}$, we write $\bm{k} \le \bm{\ell}$
if $k_{ij} \le \ell_{ij}$ for all $1 \le i \le j \le n$. To simplify
the notation we let
$$
A_{\Gamma,s}(\bm{\alpha},q):= \frac{d^s}{dq^s} \;
     A_{\Gamma}(\bm{\alpha}, q).
$$
\begin{thm}
 \label{leadingcoeffthm} 
 The quantity $A_{\Gamma,s}(\bm{\alpha},1)$ is a polynomial in the
 variables $g_{ij}$ whose homogeneous component of highest degree
 $A_{\Gamma,s}^*(\bm{\alpha},1)$ has total degree $s+|\bm{\alpha}|-1$
 and is given by
\begin{equation}
A^*_{\Gamma,s}(\bm{\alpha}, 1)
= \frac 1 {\bm{\alpha}!}\sum_{|\bm{\ell}| =
 s+ |\bm{\alpha}|-1}  C_{\Gamma,s,\bm{\ell}}^{\bm{\alpha}}\, \bm{g}^{\bm{\ell}},
\end{equation}
where
$$
C_{\Gamma,s,\bm{\ell}}^{\bm{\alpha}}:=
\frac{s!}{\bm{\ell}!} \sum_{\genfrac{}{}{0pt}{1}{\bm{k}
    \in \N^{n(n+1)/2}}{\bm{k} \le \bm{\ell}}}{S(\bm{k},\bm{\ell})\,\bm{k}! 
\sum_{\bm{p}
    \in S_{\bm{k}}}{c_{\bm{k}\bm{p}}^\alpha\,G_{\bm{p}}^{\bm{\alpha}}}}
$$
and
$$
S(\bm{\ell},\bm{k}):= \prod_{1 \le
i \le j \le n}{S(\ell_{ij}, k_{ij})}.
$$
\end{thm}

\begin{proof}
Differentiating
\[
A_{\Gamma}(\bm{\alpha}, q) = \sum_{\bm{k} \in
\N^{n(n+1)/2}}{a_{\Gamma}(\bm{\alpha}, \bm{k}, q) 
\qbinom{\bm{g}}{\bm{k}}}
\]
$s$ times with respect to $q$, we obtain
\begin{eqnarray}
&& A_{\Gamma,s}(\bm{\alpha},1) \label{sthderiv} \\ 
&=& \sum_{\bm{k} \in \N^{n(n+1)/2}}{\sum_{r=0}^s{\binom{s}{r}
\left. \left( \frac{d^r}{dq^r} \; a_{\Gamma}(\bm{\alpha}, \bm{k}, q)
\right) \left( \frac{d^{s-r}}{dq^{s-r}} \;
\qbinom{\bm{g}}{\bm{k}} \right) \right|_{q=1}}}. \nonumber 
\end{eqnarray}
By the Product Rule and Theorem~\ref{qbinder}, $\left. \left(
\frac{d^{s-r}}{dq^{s-r}} \;
\qbinom{\bm{g}}{\bm{k}} \right) \right|_{q=1}$ is a polynomial in
the variables $g_{ij}$ ($1 \le i \le j \le n$) of total degree
$|\bm{k}| + s - r$, and if $\bm{t} \in \N^{n(n+1)/2}$ with $|\bm{t}| =
s-r$, then the coefficient of $\bm{g}^{\bm{k} + \bm{t}}$ in
$\left. \left( \frac{d^{s-r}}{dq^{s-r}} \;
\qbinom{\bm{g}}{\bm{k}} \right) \right|_{q=1}$ is 
\begin{eqnarray}
\label{qbinproduct}
 \binom{s-r}{t_{11}, \dots, t_{ij}, \dots, t_{nn}} \prod_{1 \le i \le
 j \le n}{\frac{t_{ij}! S(k_{ij} + t_{ij}, k_{ij})}{(k_{ij}+t_{ij})!}}
 = \frac{(s-r)!}{(\bm{k}+\bm{t})!} \prod_{1 \le i \le j \le
 n}{S(k_{ij}+t_{ij}, k_{ij})}. 
\end{eqnarray}
By Proposition~\ref{coefderivthm}, if $r < |\bm{k}|-|\bm{\alpha}|+1$, then
$\left. \left( \frac{d^r}{dq^r} \; a_{\Gamma}(\bm{\alpha}, \bm{k}, q)
  \right) \right|_{q=1} = 0$.  Thus the total degree of each leading
term of $ A_{\Gamma,s}(\bm{\alpha},1)$ is at most
$|\bm{k}|+s-|\bm{k}|-|\bm{\alpha}|+1 = s+|\bm{\alpha}|-1$.  We now
determine the coefficient of $\bm{g}^{\bm{\ell}}$ in
$A_{\Gamma,s}(\bm{\alpha},1)$ when $\bm{\ell} \in \N^{n(n+1)/2}$
satisfies $|\bm{\ell}| = s + |\bm{\alpha}| - 1$.  A summand in
Equation~\eqref{sthderiv} has a leading term of $\bm{g}^{\bm{\ell}}$
when $\bm{k} \le \bm{\ell}$ and $r = |\bm{k}|-|\bm{\alpha}|+1 = s -
|\bm{\ell}| + |\bm{k}|$.  Given such a $\bm{k}$, let $\bm{t} =
\bm{\ell} - \bm{k}$.  Then, using Proposition~\ref{coefderivthm} and
Equation~\eqref{qbinproduct}, we find that
\begin{eqnarray}  
&& \left[ \textrm{the coefficient of } \bm{g}^{\bm{\ell}} \textrm{ in
} \right] \sum_{r=0}^s{\binom{s}{r} \left. \left( \frac{d^r}{dq^r} \;
a_{\Gamma}(\bm{\alpha}, \bm{k}, q) \right) \left(
\frac{d^{s-r}}{dq^{s-r}} \; 
\qbinom{\bm{g}}{\bm{k}}
 \right) \right|_{q=1}}
\label{leadingcoeffpartial} \\ 
&=& \binom{s}{|\bm{k}|-|\bm{\alpha}|+1} \frac{(|\bm{k}|-|\bm{\alpha}|+1)! \bm{k}!}{\bm{\alpha}!}
 \sum_{\bm{p} \in S_{\bm{k}}}
 c^{\bm{\alpha}}_{\bm{k}\bm{p}} \;
{G_{\bm{p}}^{\bm{\alpha}}}\;
\frac{(s-|\bm{k}|-|\bm{\alpha}|+1)!}{\bm{\ell}!}
S(\bm{\ell}, \bm{k}) \nonumber \\ 
&=& \frac{s!\bm{k}!}{\bm{\alpha}!\bm{\ell}!}
S(\bm{\ell}, \bm{k})
\sum_{\bm{p} \in
S_{\bm{k}}} 
 c^{\bm{\alpha}}_{\bm{k}\bm{p}} \;
{G_{\bm{p}}^{\bm{\alpha}}}. \nonumber 
\end{eqnarray}
Summing over all $\bm{k} \le \bm{\ell}$ completes the proof.
\end{proof}

As a special case of this theorem, we can consider the quiver $S_g$
from Section~\ref{introsec}.  In this case,
$A_{\Gamma,s}(\bm{\alpha},1)$ is a polynomial in $g$ of degree $s +
\alpha - 1$ and leading coefficient
\[
\frac{s!}{\alpha! (s+\alpha-1)!} \sum_{k =
  \alpha-1}^{s+\alpha-1}{S(s+\alpha-1, k)\,k!
  \sum_{p=\alpha-1}^k{G^{\alpha}_p \sum_{j=0}^{\infty}{ \binom{p}{j}
      \binom{\alpha}{k-p-j} 2^{p-j}}}}.
\] 

\appendix
\section{Derivatives of $q$-binomial coefficients}
\label{appendixsec}

The goal of this appendix is to prove a theorem about derivatives of
$q$-binomial coefficients.  The $q$-binomial coefficient
$\qbinom{b}{k}$ is a polynomial in $q$ of degree $k(b-k)$ defined for
non-negative integers $b$ and $k$ by the formula 
\[
  \qbinom{b}{k} = \frac{(q^b-1)(q^{b-1}-1) \cdots (q^{b-k+1}-1)}
                       {(q^k-1)(q^{k-1}-1) \cdots (q-1)}
                = \prod_{i=1}^k{
                    \frac{\sum_{j=0}^{b-i}{q^j}}
                         {\sum_{j=0}^{i-1}{q^j}}
                  }.
\]
The result in question involves the Stirling number of the second kind
$S(\ell, k)$, which counts the number of partitions of an $\ell$-set
into $k$ blocks.  While there are many formulas involving Stirling
numbers of the second kind, we will only need to know that 
\begin{eqnarray}  
\label{stirlinggf}
  \sum_{\ell = k}^{\infty}{S(\ell,k) \; \frac{x^{\ell}}{\ell!}} &=&
  \frac{(e^x-1)^k}{k!} 
\end{eqnarray}
for all $k \ge 0$.  Now we can state the theorem.

\begin{thm} \label{qbinder}
Fix non-negative integers $k$ and $t$.  Then there is a polynomial
$P_{k,t}(b)$ of degree $k+t$ with leading coefficient
$\frac{t!}{(k+t)!} \cdot S(k+t, k)$ so that for all non-negative
integers $b$, 
\[
  P_{k,t}(b) = \left. \left( \frac{d^t}{dq^t} \; \qbinom{b}{k} \right) \right|_{q=1}.
\]
\end{thm}

We need one lemma before proving Theorem~\ref{qbinder}.

\begin{lem} \label{qbinlem}
Fix non-negative integers $i$ and $m$.  Then there is a polynomial
$p_{i,m}(b)$ of degree $m+1$ with leading coefficient
$\frac{1}{i(m+1)}$ so that for all positive integers $b \ge i-1$, 
\[
p_{i,m}(b) = \left. \left( \frac{d^m}{dq^m} \;
\frac{\sum_{j=0}^{b-i}{q^j}}{\sum_{j=0}^{i-1}{q^j}} \right)
\right|_{q=1}. 
\]
\end{lem}

\begin{proof}
Let $r$ be a non-negative integer and define the function
\[
 f_{m,r}(b,i,q) = \frac{d^m}{dq^m} \frac{\frac{d^r}{dq^r}
 \sum_{j=0}^{b-i}}{\sum_{j=0}^{i-1}{q^j}}. 
\]
We want to show that $f_{m,0}(b,i,1)$ is a polynomial of degree $m+1$
with leading coefficent $\frac{1}{i(m+1)}$, but by induction on $m$ we
will prove the stronger statement that $f_{m,r}(b,i,1)$ is a
polynomial of degree $m+r+1$ with leading coefficent
$\frac{1}{(m+r+1)}$.  When $m = 0$, we find that 
\begin{eqnarray*}
f_{0,r}(b, i, q) &=& \frac{\frac{d^r}{dq^r} \;
\sum_{j=0}^{b-i}{q^j}}{\sum_{j=0}^{i-1}{q^j}} =
\frac{\sum_{j=r}^{b-i}{j(j-1) \cdots (j-r+1)
q^{j-r}}}{\sum_{j=0}^{i-1}{q^j}} \\ 
f_{0,r}(b, i, 1) &=& \frac{\sum_{j=r}^{b-i}{j(j-1) \cdots (j-r+1)}}{i}
= \frac{r!}{i} \binom{b-i+1}{r+1}. 
\end{eqnarray*}
Thus $f_{0,r}(b,i,1)$ is a polynomial in $b$ of degree $r+1$ and
leading coefficient $\frac{1}{i(r+1)}$, proving the base case for the
induction.  When $m \ge 1$, writing the Quotient Rule in the form 
\[
\frac{d}{dq} \left( \frac{g(q)}{h(q)} \right) = \frac{g'(q)}{h(q)} -
\frac{g(q)}{h(q)} \cdot \frac{h'(q)}{h(q)}, 
\]
we see that
\[
 f_{1,r}(b,i,q) = f_{0,r+1}(b,i,q) - f_{0,r}(b,i,q) \cdot
 \frac{\sum_{j=0}^{i-1}{jq^{j-1}}}{\sum_{j=0}^{i-1}{q^j}}. 
\]
Differentiating both sides $m-1$ times and evaluating at $q = 1$, we obtain
\[
f_{m,r}(b,i,1) = f_{m-1,r+1}(b,i,1) - \sum_{s=0}^{m-1}{\binom{m-1}{s}
f_{s,r}(b,i,1) \left( \frac{d^{m-1-s}}{dq^{m-1-s}} \;
\frac{\sum_{j=0}^{i-1}{jq^{j-1}}}{\sum_{j=0}^{i-1}{q^j}} \right)}. 
\]
The inductive hypothesis shows that $f_{m-1,r+1}(b,i,1)$ is a
polynomial in $b$ of degree $m+r+1$ and leading coefficient
$\frac{1}{i(m+r+1)}$, and every other term on the right-hand side of
the equation is a polynomial in $b$ of degree less than $m+r+1$.  This
proves that $f_{m,r}(b,i,1)$ is a polynomial in $b$ of degree $m+r+1$
and leading coefficient $\frac{1}{i(m+r+1)}$. 
\end{proof}

Now we can prove Theorem~\ref{qbinder}.

\begin{proof}
We link derivatives of $\qbinom{b}{k}$ with the derivatives computed
in Lemma~\ref{qbinlem} via Taylor series expansions at $q = 1$: 
\begin{eqnarray*} 
\sum_{t=0}^{\infty}{\left. \left( \frac{d^t}{dq^t} \qbinom{b}{k}
\right) \right|_{q=1} \; \frac{(q-1)^t}{t!}}  
&=& \qbinom{b}{k} \\
&=& \prod_{i=1}^k{\frac{\sum_{j=0}^{b-i}{q^j}}{\sum_{j=0}^{i-1}}} \\
&=& \prod_{i=1}^k{\left( \sum_{m=0}^{\infty}{\left. \left(
\frac{d^m}{dq^m} \;
\frac{\sum_{j=0}^{b-i}{q^j}}{\sum_{j=0}^{i-1}{q^j}} \right)
\right|_{q=1} \frac{(q-1)^m}{m!}} \right)}. 
\end{eqnarray*}
By Lemma~\ref{qbinlem}, for all integers $b \ge k-1$, 
\[
 \prod_{i=1}^k{\left( \sum_{m=0}^{\infty}{\left. \left(
 \frac{d^m}{dq^m} \;
 \frac{\sum_{j=0}^{b-i}{q^j}}{\sum_{j=0}^{i-1}{q^j}} \right)
 \right|_{q=1} \frac{(q-1)^m}{m!}}\right)} = \prod_{i=1}^k{\left(
 \sum_{m=0}^{\infty}{p_{i,m}(b) \; \frac{(q-1)^m}{m!}}\right)}, 
\]
where $p_{i,m}(b)$ is the polynomial from Lemma~\ref{qbinlem}.  In
fact, we claim this equation holds for $0 \le b < k-1$ as well.  Fix
$b$ with $0 \le b < k-1$.  The left-hand side is certainly 0, and by
Lemma~\ref{qbinlem}, $p_{b+1,m}(b) = 0$ for all $m$, so the right-hand
side is 0 as well.  This shows that for all non-negative integers $b$,  
\[ 
\sum_{t=0}^{\infty}{\left. \left( \frac{d^t}{dq^t} \qbinom{b}{k}
\right) \right|_{q=1} \frac{(q-1)^t}{t!}}  
= \prod_{i=1}^k{\left( \sum_{m=0}^{\infty}{p_{i,m}(b)
\frac{(q-1)^m}{m!}}\right)}. 
\]
Thus there is a polynomial $P_{k,t}(b)$ in $b$ of degree $k+t$ so that
$P_{k,t}(b) = \left. \left( \frac{d^t}{dq^t} \qbinom{b}{k} \right)
\right|_{q=1}$ for all non-negative integers $b$, and its leading
coefficient is  
\begin{eqnarray*}
&& \left[ \textrm{coefficient of $\frac{(q-1)^t}{t!}$} \right]
\prod_{i=1}^k{\left( \sum_{m=0}^{\infty}{\frac{1}{i(m+1)}
\frac{(q-1)^m}{m!}}\right)} \\ 
&=& \left[ \textrm{coefficient of $\frac{(q-1)^t}{t!}$} \right]
\prod_{i=1}^k{\frac{1}{i} \left( \frac{e^{q-1}-1}{q-1} \right)} \\ 
&=& \left[ \textrm{coefficient of $\frac{(q-1)^t}{t!}$} \right]
\frac{1}{k!} \left( \frac{e^{q-1}-1}{q-1} \right)^k \\ 
&=& \left[ \textrm{coefficient of $\frac{(q-1)^t}{t!}$} \right]
\sum_{m=0}^{\infty}{\frac{m!}{(k+m)!} \; S(k+m, k) \;
\frac{(q-1)^m}{m!}} \\ 
&=& \frac{t!}{(k+t)!} \; S(k+t, k),
\end{eqnarray*}
where the second-to-last line comes from the generating function in 
Equation~\eqref{stirlinggf}. 
\end{proof}

\bibliographystyle{amsplain}

\providecommand{\bysame}{\leavevmode\hbox to3em{\hrulefill}\thinspace}
\providecommand{\MR}{\relax\ifhmode\unskip\space\fi MR }
\providecommand{\MRhref}[2]{%
  \href{http://www.ams.org/mathscinet-getitem?mr=#1}{#2}
}
\providecommand{\href}[2]{#2}

\end{document}